\documentclass{amsart}
\usepackage{amsmath}
\usepackage{amsfonts}
\usepackage{amssymb}
\usepackage{amsthm}
\usepackage{color,hyperref}
\usepackage[table,usenames,dvipsnames]{xcolor}
\usepackage{geometry} 
\usepackage{tikz}
\usetikzlibrary{matrix,arrows,positioning,automata}
\pgfdeclarelayer{background layer}
\pgfsetlayers{background layer,main}

\definecolor{gray9}{gray}{0.9}
\definecolor{gray8}{gray}{0.8}
\definecolor{gray7}{gray}{0.7}
\definecolor{gray6}{gray}{0.6}
\definecolor{lgr}{rgb}{0.8,0.8,0.8}

\definecolor{darkblue}{rgb}{0.0,0.0,0.3}
\hypersetup{colorlinks,breaklinks,
            linkcolor=darkblue,urlcolor=darkblue,
            anchorcolor=darkblue,citecolor=darkblue}

\usepackage[norelsize,ruled,vlined,linesnumbered]{algorithm2e}

\newcommand{\cT}{{\mathcal T}}
\newcommand{\cS}{{\mathcal S}}
\newcommand{\Sub}{\mathbf{Sub}}
\newcommand{\Miss}{\mathbf{m}}
\newcommand{\Closure}{\mathbf{Cl}}

\newcommand{\Max}{\mathbf{Max}}
\newcommand{\Aut}{\mathbf{Aut}}

\newcommand{\Z}{\mathbb{Z}}

\newcommand{\cl}{\color{lgr}}

\usepackage{cite}

\DeclareMathOperator{\im}{im}
\DeclareMathOperator{\id}{id}

\theoremstyle{plain}
\newtheorem{theorem}{Theorem}[section]
\newtheorem{lemma}[theorem]{Lemma}

\theoremstyle{definition}

\newtheorem{problem}[theorem]{Problem}

\newcommand{\GAP}{\textsc{Gap}}

\newcommand{\Semigroups}{\textsc{Semigroups}}
\newcommand{\SubSemi}{\textsc{SubSemi}}
\newcommand{\SmallGroup}{\textsc{SmallGroup}}

\begin{document}
\title{Enumerating Transformation Semigroups}
\author[J. East, A. Egri-Nagy, J. D. Mitchell]{James East$^1$ and Attila Egri-Nagy$^{2}$ and James D. Mitchell$^3$}
\address{$^1$Centre for Research in Mathematics, School of Computing, Engineering and Mathematics, University of Western Sydney (Parramatta Campus), Locked Bag 1797, Penrith, NSW 2751, Australia}
\address{$^2$  Akita International University, Yuwa, Akita-City, 010-1292, Japan}
\address{$^3$ Mathematical Institute, University of St Andrews, North Haugh, St Andrews, Fife, KY16 9SS, Scotland}
\email{J.East@uws.edu.au,\ egri-nagy@aiu.ac.jp,\ jdm3@st-and.ac.uk}
\maketitle
\begin{abstract}
We describe general methods for enumerating subsemigroups of finite semigroups and techniques to improve the algorithmic efficiency of the calculations.
As a particular application we use our algorithms to enumerate all transformation semigroups up to degree 4.
Classification of these semigroups up to conjugacy, isomorphism and anti-isomorphism, by
size and rank, provides a solid base for further investigations of
transformation semigroups.
\end{abstract}

\section{Introduction}
When studying finite structures it can be helpful to generate small
examples using a computer \cite{Erne1991,Pfeiffer2004,Harrison1965,Harary1973,SmallGroupsLibrary,holt2010enumerating}.
By investigating these sample objects we can discover patterns, formulate new hypotheses, look for counterexamples, and so on.
More diverse sample sets make it easier to greater evidence for or against a conjecture.
Therefore, to maximize the usefulness of these small examples we naturally aim both to enumerate \emph{all} objects of a certain
parameter value (such as size or dimension), and to \emph{increase} the value of this parameter.

Earliest efforts to enumerate semigroups have focused on the abstract case, enumerating by \emph{order} (the \emph{size} of the semigroup).
The basic idea for enumerating by size is to find all valid multiplication tables (up to isomorphism and anti-isomorphism) of the given size \cite{monoidenum2009,For55,JW77,KRS76,Ple67,SZT94,tamura2,tamura1,Distler2014}.
Our approach here is to enumerate not by size but, rather, by
\emph{degree} of transformation representation \cite{T3enum1970,T3enum1991}.  Recall that Cayley's Theorem (for semigroups) says that any finite semigroup $S$ is isomorphic to a semigroup of functions on a finite set; the \emph{degree} of $S$ is defined to be the minimal size of such a set.  So we aim to enumerate the valid subtables inside the multiplication table of the \emph{full transformation semigroup} $\cT_n$, which consists of all \emph{transformations} of the set $\{1,\ldots,n\}$ (i.e., all functions from this set to itself); more specifically, we wish to find all such subtables (up to isomorphism and anti-isomorphism) that are not also (isomorphic to) subtables of $\cT_{n-1}$.  All finite semigroups will eventually be listed when we enumerate by size or by degree, but the order of the list is different from one method to the other.  For example, there are 52,989,400,714,478 abstract semigroups of order 9
\cite{monoidenum2009, Distler2014},
 so one could barely imagine the number of semigroups of
order 27, where $\cT_3$ would first appear when enumerating by size.  On the other hand, our results show that there are only~25 different transformation semigroups on~3 points of order~9.
Metaphorically speaking, enumeration by size and by degree go in completely different
directions and proceed at a different rate.

The complexity of the multiplication of two transformations from $\cT_n$ is linear in $n$.
However, multiplication in a semigroup defined by a multiplication table has constant complexity.
Hence, we choose multiplication tables as the main way of representing semigroups.
This decision has two consequences.
First, our algorithms fall into the class of semigroup algorithms that fully enumerate the elements.
This of course  restricts us to relatively small semigroups.
On the other hand, multiplication table algorithms are completely representation independent,
so our methods are widely applicable across different types of finite semigroups.

The article is organised as follows.
In Section \ref{sec:multab} we describe a basic multiplication table method to calculate subsemigroups generated by a subset.
In Section \ref{sec:enum} we present generic search algorithms to enumerate subsemigroups of
finite semigroups.
In Section \ref{sec:techniques} we discuss techniques for improving the
efficiency of the algorithms in Section~\ref{sec:enum}, based on more specific algebraic results.
Finally, in Section \ref{sec:fulltranssgp} we apply the developed methods for enumerating transformation semigroups acting on up to 4 points.

For the computational enumeration we used the \GAP~computer algebra system \cite{GAP4} and its \Semigroups~package~\cite{Semigroups} and developed a new package \SubSemi~for subsemigroup enumeration \cite{subsemi}. Instructions for recomputing the results can be found in the package documentation.  We now take a moment to establish some notation.

\subsection{Notation}

Let $S$ be a finite semigroup with $|S|=n\in\mathbb{N}$.
We fix an order on the semigroup elements and for convenience denote them simply by $1,\ldots,n$.
The  \emph{multiplication table}, or \emph{Cayley table} of $S$ is an $n\times n$ matrix $M_S$ with entries from
$\{1,\ldots,n\}$, such that $M_{i,j}=k$ if $ij=k$ in $S$ (we denote multiplication in $S$ by juxtaposition).
The subarray of $M_S$ with rows and columns indexed by a subset $A\subseteq\{1,\ldots,n\}$ is denoted by $M_A$.

The set of all subsemigroups of $S$ is denoted by $\Sub(S)=\big\{T\mid T\leq S
\big\}$.  We consider the empty set a semigroup, so
$\varnothing\in\Sub(S)$.  The set of maximal proper subsemigroups
of $S$ is denoted by $\Max(S)$.  For $A\subseteq S$, $\langle A\rangle$ denotes the least subsemigroup of $S$
containing $A$, the semigroup generated by $A$.

If $I$ is an ideal of $S$ then the \emph{Rees factor semigroup} $S/I$ has
elements $(S\setminus I)\cup\{0\}$, where $0$ is a new symbol that does not belong to $S$, and with multiplication $\cdot$ defined by
\[
s\cdot t = \begin{cases}
st &\text{if $s,t,st\in S\setminus I$}\\
0 &\text{otherwise.}
\end{cases}
\]
Note that if $I=\varnothing$ is the empty ideal, then $S/I\cong S^0$ is the semigroup obtained from $S$ by adjoining a zero.

A transformation $t\in\cT_n$ will often be written by simply listing the images of the points, $t=[t(1),\ldots,t(n)]$.  The group consisting of all permutations of $\{1,\ldots,n\}$ is the \emph{symmetric group}, denoted~$\cS_n$, and is the group of units of~$\cT_n$.

\section{The closure algorithm}
\label{sec:multab}

A basic question in computational semigroup theory is:  \emph{Given a subset $A$ of a semigroup $S$, what is the subsemigroup $\langle A\rangle$ generated by $A$?}  We will also write $\Closure(A)$ for $\langle A\rangle$, and refer to it as the \emph{closure} of $A$ (in $S$).
Note that on the level of the Cayley table, $\langle A\rangle$ is the minimal set $B$ (in the containment order on subsets of~$S$) such that the subarray $M_B$ contains $M_A$ and all entries from $M_B$ belong to $B$.  An algorithm for obtaining $B=\Closure(A)$ is as follows.
First we determine the set $\Miss(A)=\{x\mid x\text{ is an entry of } M_A\}\setminus A$, those products of elements of
$A$ that are missing from $A$.
We then recursively define
\[
\Closure_1(A)=A\cup\Miss(A),\qquad \Closure_{i+1}(A)=\Closure_1(\Closure_{i}(A))\quad \text{for $i\geq 1$.}
\]
Note that $\Closure(A)=\Closure_j(A)$, where $j$ is minimal such that $\Closure_j(A)=\Closure_{j+1}(A)$, or equivalently $\Miss(\Closure_j(A))=\varnothing$.

For example, consider the symmetric group $\cS_3$, with its elements ordered $1=()$, $2=(2,3)$, $3=(1,2)$, $4=(1,2,3)$, $5=(1,3,2)$, $6=(1,3)$.
The subarray $M_{\{4\}}$ contains only one entry, 5, which is different from~4, so the subarray is not closed. With the above notation, $\Miss(\{4\})=\{5\}$, $\Miss(\{4,5\})=\{1\}$ and $\Miss(\{1,4,5\})=\varnothing$, so $\Closure_1(\{4\})=\{4,5\}$ and $\Closure_2(\{4\})=\Closure_3(\{4\})=\Closure(\{4\})=\{1,4,5\}$, corresponding to the unique subgroup of order $3$ in $\cS_3$.  See Fig.~\ref{fig:S3}.

\begin{figure}[ht]
\begin{center}
\renewcommand{\arraystretch}{0.5}

\begin{tabular}{@{}c@{\;\!}|@{\;\!}c@{}c@{}c@{}c@{}c@{}c@{}}
  &1&2&3&4&5&6\\
\hline
1&1&2&3&4&5&6\rule{0pt}{2ex}\\
2&2&1&4&3&6&5\\
3&3&5&1&6&2&4\\
4&4&6&2&5&1&3\\
5&5&3&6&1&4&2\\
6&6&4&5&2&3&1\\
\end{tabular}
\qquad\qquad\qquad
\begin{tabular}{@{}c@{\;\!}|@{\;\!}c@{}c@{}c@{}c@{}c@{}c@{}}
  &\cl1&\cl2&\cl3&   4&\cl5&\cl6\\
\hline
\cl1&\cl1&\cl2&\cl3&\cl4&\cl5&\cl6\rule{0pt}{2ex}\\
\cl2&\cl2&\cl1&\cl4&\cl3&\cl6&\cl5\\
\cl3&\cl3&\cl5&\cl1&\cl6&\cl2&\cl4\\
   4&\cl4&\cl6&\cl2&   5&\cl1&\cl3\\
\cl5&\cl5&\cl3&\cl6&\cl1&\cl4&\cl2\\
\cl6&\cl6&\cl4&\cl5&\cl2&\cl3&\cl1\\
\end{tabular}
\quad
$\longrightarrow$
\quad
\begin{tabular}{@{}c@{\;\!}|@{\;\!}c@{}c@{}c@{}c@{}c@{}c@{}}
  &\cl1&\cl2&\cl3&   4&   5&\cl6\\
\hline
\cl1&\cl1&\cl2&\cl3&\cl4&\cl5&\cl6\rule{0pt}{2ex}\\
\cl2&\cl2&\cl1&\cl4&\cl3&\cl6&\cl5\\
\cl3&\cl3&\cl5&\cl1&\cl6&\cl2&\cl4\\
   4&\cl4&\cl6&\cl2&   5&   1&\cl3\\
   5&\cl5&\cl3&\cl6&   1&   4&\cl2\\
\cl6&\cl6&\cl4&\cl5&\cl2&\cl3&\cl1\\
\end{tabular}
\quad
$\longrightarrow$
\quad
\begin{tabular}{@{}c@{\;\!}|@{\;\!}c@{}c@{}c@{}c@{}c@{}c@{}}
  &   1&\cl2&\cl3&   4&   5&\cl6\\
\hline
   1&   1&\cl2&\cl3&   4&   5&\cl6\rule{0pt}{2ex}\\
\cl2&\cl2&\cl1&\cl4&\cl3&\cl6&\cl5\\
\cl3&\cl3&\cl5&\cl1&\cl6&\cl2&\cl4\\
   4&   4&\cl6&\cl2&   5&   1&\cl3\\
   5&   5&\cl3&\cl6&   1&   4&\cl2\\
\cl6&\cl6&\cl4&\cl5&\cl2&\cl3&\cl1\\
\end{tabular}
    \caption{The Cayley table of the symmetric group $\cS_3$ (left), and calculation of $\Closure(\{4\})=\{1,4,5\}$ (right).  See text for further explanation.}
    \label{fig:S3}
   \end{center}
 \end{figure}

The above recursive definition describes an algorithm for calculating the closure, but not an efficient one.
We can avoid the full calculation of the missing elements in the recursive steps.
When extending the subarray $M_A$ by a single element $i$, if $\Miss(A)$ is already calculated, then all new missing elements in $\Miss(A\cup\{i\})$ can only come from the $i$th row or the $i$th column.
So, to calculate the closure we can extend the subarray one-by-one using the elements of $\Miss(A)$ and any new missing elements encountered during the recursion.
This way each table entry is checked only once.

\section{Basic Search Algorithms for Subsemigroup Enumeration}
\label{sec:enum}

One of our primary goals is to enumerate the semigroups of a given degree, and this involves enumerating the subsemigroups of $\cT_n$.  This is of course a special instance of the following more general problem.

\begin{problem}\label{prob:enum}
For a semigroup $S$, find all of its subsemigroups:
$\Sub(S)=\left\{ T\mid T\leq S\right\}.$
\end{problem}

In this section, we discuss a number of algorithmic approaches to this
problem.  Thinking in terms of the multiplication table $M_S$, we are
looking for all subarrays $M_A$ that are also multiplication tables;
i.e., they do not contain elements not in $A$.

\subsection{Enumerating by Brute Force}
\label{sec:bruteforce}
The obvious brute-force algorithm for constructing $\Sub(S)$ proceeds by first constructing the powerset $2^S=\left\{ A\mid A\subseteq S\right\}$ and then checking each subset for closure.  For some semigroups, any method essentially reduces to the brute-force algorithm (e.g., left or right zero semigroups, where every subset is a subsemigroup), but it is inefficient in cases where only a fraction of the subsets are closed under multiplication.

\subsection{Enumerating by Minimal Generating Sets}
\label{sec:mingen}
The \emph{rank} of a semigroup is the least size of a  generating set.
The rank of a subsemigroup can be bigger than the rank of the semigroup itself.
For example, the full transformation semigroup has rank 3~\cite{ClassicalTransSemigroups2009}, but its minimal ideal (which is a left zero semigroup) has rank $n$, while its maximal proper ideal (the semigroup of all singular transformations) has rank ${n\choose2}=n(n-1)/2$ for $n\geq3$ \cite{Gomes1987,Howie1978}.
Assuming that we know the maximum rank for subsemigroups of $S$, we can check all subsets of $S$ with cardinality up to that value to see what subsemigroups they generate.
The same subsemigroup may be generated by many generating sets but the maximality guarantees that we construct all of $\Sub(S)$.
On each level $k$, we check $\binom{|S|}{k}$ many generating sets.
Therefore, the method is only feasible if the maximum value of the ranks of the subsemigroups is known to be small.
To the knowledge of the authors, the maximal rank of a subsemigroup of $\cT_n$ is not currently known, though the maximal rank of an ideal is equal to the largest of the Stirling numbers $S(n,2),\ldots,S(n,n-1)$ \cite{Howie1990}. 

What can we do if we do not know the maximum rank value?
We can keep going until no new subsemigroup is generated.
First we check all subsemigroups generated by one element, then those generated by two; we subtract the former set from the latter to obtain the set of rank-2 subsemigroups.
We then continue up to $m$ where the set of rank-$m$ subsemigroups is empty.
Unfortunately this last step is wasted, unless rank$(S)=S$, e.g.~left zero
semigroups, in which case, this is just the brute-force search.

\subsection{Enumerating by Minimal Extensions}
\label{sec:minext}

A \emph{minimal extension} of a subsemigroup $T\leq S$ is a subsemigroup $\langle T\cup\{u\}\rangle$, where $u\in S\setminus T$.
We simply add a new element to $T$ and calculate the closure.
If we recursively calculate minimal extensions, then we obtain all subsemigroups of $S$ containing $T$.
So this represents a solution to a natural generalisation of Problem \ref{prob:enum}, that of calculating the interval $[T,S]=\left\{U\mid T\leq U\leq S\right\}$ in the lattice $\Sub(S)$.

This algorithm is a graph search.
The nodes are the subsemigroups containing $T$, and there is a directed edge labelled $u$ from $V$ to $V'$ if $V'=\langle V\cup\{u\}\rangle$.
In general, there may be many incoming edges to a subsemigroup.
The efficiency of the algorithm  comes from the fact that the search tree is cut when the search encounters a subsemigroup already known, simply by making no further extensions.
The details are provided in Algorithm \ref{alg:minclosure}.
Depending on how \textsf{exts}, the storage for extensions, behaves under the \texttt{Store}/\texttt{Retrieve} operations we get different search strategies. If \textsf{exts} is a stack, then the algorithm performs  a depth-first search, and if it is a queue, then a breadth-first search is performed.
A full subsemigroup enumeration can be done by starting the algorithm with parameters $T=\varnothing$, the semigroup is simply $S$, and  $X=S$.
This is simply extending the empty set by all elements of $S$ recursively.

When using the breadth-first search strategy, the generating set is minimal, so Algorithm \ref{alg:minclosure} can easily be modified to enumerate minimal generating sets.
A little consideration shows that this is a more efficient version of the minimal generating sets algorithm (Section \ref{sec:mingen}), but it does not escape checking generating sets one bigger than the maximal rank.

\begin{algorithm}[ht]
\SetKwInOut{Input}{input}
\SetKwInOut{Output}{output}
\SetKwData{subs}{subs}
\SetKwData{gens}{gens}
\SetKwData{exts}{exts}
\SetKwFunction{Store}{Store}
\SetKwFunction{Retrieve}{Retrieve}
\SetKwFunction{MinimalExtensions}{SubSemigroupsByMinimalExtensions}
\Input{$S$, the semigroup\\ $T\leq S$, a subsemigroup to be extended\\$X\subseteq S$, a set of elements to extend with}
\Output{all $T'\subseteq S$ such that $T'=\langle T\cup Y\rangle$ for some $Y\subseteq X$}
\SetKwInOut{Name}{\MinimalExtensions($T$,$S$,$X$)}
\BlankLine
\Name{}
\subs $\leftarrow \{T\}$\\
\exts $\leftarrow \varnothing$\\
      \For{$s\in$ $(S\setminus T)\cap X$}{
      \Store(\exts, $T\cup\{s\}$)
    }
\While{$|\exts|>0$}{
 $T'\leftarrow\langle\Retrieve(\exts)\rangle$\\
 \If{$T'\notin$ \subs}{
   \subs$\leftarrow$ \subs$\cup\ \{T'\}$\\
      \For{$s\in$ $(S\setminus T')\cap X$}{
      \Store(\exts, $T'\cup\{s\}$)
      }
  }
}
\Return \subs
\caption{Finding subsemigroups by minimal extensions.}
\label{alg:minclosure}
\end{algorithm}

\subsection{Enumerating by Maximal Subsemigroups}
Assuming that we have calculated the maximal subsemigroups of $S$, we can parallelize subsemigroup enumeration by enumerating subsemigroups of the maximal subsemigroups and merging the results, using the obvious fact that
\[
\Sub(S)=\{S\} \cup  \bigcup_{T\in \Max(S)}\Sub(T).
\]
A description is given in \cite{MaxSubSemi} of the possible maximal subsemigroups of an arbitrary finite semigroup $S$, i.e.~a maximal subsemigroup of $S$ is one of the possibilities described in \cite{MaxSubSemi}.
Algorithms, based on these properties, are implemented in the \Semigroups~package \cite{Semigroups,AlgMaxSemi}.
The sets of subsemigroups of the maximal subsemigroups do overlap in general, so the same subsemigroup may get enumerated many times, and merging is a nontrivial step.
Note that recursively iterating the maximal subsemigroups is a variant of the depth-first search algorithm, moving from large subsemigroups to small in contrast to the Minimal Extensions method discussed in Section \ref{sec:mingen}.

\section{Advanced Algorithmic Techniques}
\label{sec:techniques}

Since we are dealing with well-studied algebraic structures, there are many mathematical results we may exploit in order to improve the efficiency of any basic subsemigroup enumeration algorithm.  In this section, we outline a number of improvements on the algorithms described in the previous section, the most powerful of which involves using an ideal to parallelize the enumeration (see Sections \ref{sec:idealparallel} and \ref{sec:lowertorso}).

\subsection{Equivalent Generators}
\label{sec:equivgen}
We define an equivalence relation $\equiv$ on $S$ by
$$ s\equiv t \Longleftrightarrow \langle s \rangle= \langle t  \rangle.$$
For instance, $[ 2, 3, 1, 1 ]\equiv[ 3, 1, 2, 2 ]$ in $\cT_4$, both transformations generating $\big\{ [ 2, 3, 1, 1 ], [ 3, 1, 2, 2 ], [ 1, 2, 3, 3 ]\big\}$.
Note that if the $\equiv$-class of $s\in S$ is nontrivial, then $\langle s\rangle$ is a cyclic group.
Note also that if $s\equiv t$, then $\langle U\cup\{s\}\rangle=\langle U\cup\{t\}\rangle$ for any $U\subseteq S$, so when using the Minimal Extensions enumeration method (Section~\ref{sec:mingen}), we only need to calculate extensions with respect to $\equiv$-class representatives.

\subsection{Exploiting Symmetries}

If we have information about the symmetries $S$ might possess, then we can accelerate any subsemigroup enumeration algorithm.  For example, if $S$ is a monoid with group of units $G$, then $T^g=\left\{ g^{-1}tg \mid t\in T \right\}$ is a subsemigroup of $S$ for any $T\in\Sub(S)$ and $g\in G$.  More generally, if $\theta$ is an automorphism of $S$, then $T\theta=\left\{ t\theta \mid t\in T \right\}\in\Sub(S)$ for all $T\in\Sub(S)$.  There is an algorithm for computing the automorphism group $\Aut(S)$ of a finite semigroup $S$ \cite{computingautomorphisms2010}.  So, during subsemigroup enumeration, whenever a subsemigroup $T\leq S$ is found, we may quickly find the subsemigroups $T\theta$.
If $H\leq\Aut(S)$ is a group of automorphisms of $S$, then we write $\Sub_H(S)$ for a set of automorphism class representatives of the subsemigroups of $S$; so for any $T\in\Sub(S)$, there is a unique subsemigroup $U\in\Sub_H(S)$ such that $T=U\theta$ for some $\theta\in H$.

Moreover, when extending a subsemigroup $T$ by elements from $S\setminus T$ one by one (Section \ref{sec:minext}), we can cut the search tree further. By taking the normalizer of $T$ in $S$, i.e.\ the stabilizer of $T$ under the conjugation, we know that the conjugacy classes of $T\cup\{x\}$ will be of the form $T\cup\{x_i\}$, where $x_i$ is an element of the orbit of $x$ under conjugation, thus we only need to extend with one element of the orbit.
Therefore, the search algorithm has to visit  only the orbit representatives of $S\setminus T$ under the normalizer.

\subsection{Parallel Enumeration in Ideals and Rees Quotients}
\label{sec:idealparallel}

In general, an ideal $I$ of $S$ divides a subsemigroup $T\leq S$ into two parts: a subsemigroup contained in the ideal, $L=T\cap I$, and a subset outside the ideal, $U=T\cap (S\setminus I)$. We call $L$ and $U$ the \emph{lower} and \emph{upper torso of $T$ with respect to $I$}, respectively (see Fig.~\ref{fig:torsos}).  Note that $U$ or $L$ may be empty.
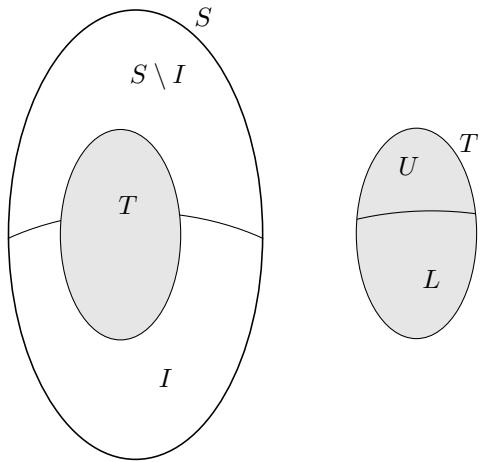
\begin{figure}[ht]
\begin{center}
\def\S{(0,0) ellipse (1.7cm and 3cm)}
\def\I{(0,-1.7) ellipse (3cm and 2cm)}
\def\T{(-.2,0)ellipse (.8cm and 1.4cm)}
\begin{tikzpicture}
\begin{scope}
\clip\S;
\draw\I;
\draw[very thick]\S;
\end{scope}
\fill [color=gray9] \T;
\draw \T;
\draw (0.3,2.1) node (SI) {$S\setminus I$};
\draw (0.9,2.9) node (S) {$S$};
\draw (0.4,-1.9) node (I) {$I$};
\draw (-0.1,.4) node (T) {$T$};
\end{tikzpicture}
\begin{tikzpicture}
\begin{scope}
\clip\T;
\draw[very thick]\S;
\fill [color=gray9] \T;
\draw\I;
\end{scope}
\draw \T;
\draw (.5,1.2) node (T) {$T$};
\draw (0,-.6) node (L) {$L$};
\draw (-0.3,0.9) node (U) {$U$};
\draw (-2,-2.9) node (invisible) {};
\end{tikzpicture}
\caption{If $S$ has an ideal $I$, then a subsemigroup $T\leq S$ is partitioned into two parts by the ideal, the upper torso $U=T\cap (S\setminus I)$ and the lower torso $L=T\cap I$.}
\label{fig:torsos}
\end{center}
\end{figure}
The upper torso $U$ need not be a subsemigroup of $T$ in general, but $U$ may be turned into a semigroup by adjoining a zero element in an obvious way, giving precisely the Rees quotient of $T$ by the ideal $T\cap I$.
The next lemma shows that subsemigroup enumeration can be done in parallel in $I$ and $S/I$, though a combination step is still required.

\begin{lemma}
\label{lem:torso}
If $I$ is an ideal of $S$, then $$\Sub(S)=\big\{\langle (U\setminus\{0\})\cup L \rangle\mid U\in \Sub(S/I),\ L\in\Sub(I)\big\}.$$
\end{lemma}
\begin{proof}
Let $T\in\Sub(S)$ and put $L=T\cap I$ and $U=T\cap(S\setminus I)\cup\{0\}$. Then $T=(U\setminus \{0\})\cup L=\langle (U\setminus\{0\})\cup L\rangle$. This establishes the forward set containment. The other is trivial.
\end{proof}

Note that a subsemigroup $U\leq S/I$ need not contain the zero element.  The method suggested by Lemma~\ref{lem:torso} requires calculating $|\Sub(S/I)|\cdot|\Sub(I)|$ set unions and subsequent closures.
However, this inefficient calculation is avoidable by using the Lower Torso Enumeration technique, as we now describe.

\subsection{Lower Torso Enumeration}
\label{sec:lowertorso}

Suppose we have enumerated $\Sub(S/I)$ for some ideal $I$ of the semigroup~$S$.  Then $\left\{ U\setminus\{0\} \mid U\in\Sub(S/I) \right\}$ is precisely the set of all upper torsos.
The next task is to find all the matching lower torsos for an upper torso.  That is, for each upper torso $U$, we must find all subsemigroups $L\leq I$ such that $U\cup L\leq S$.
The method described in Section \ref{sec:idealparallel} involved enumerating $\Sub(S/I)$ and $\Sub(I)$ and checking what the combinations generated.
We can do better.
The idea is that an upper torso $U$ acts on the elements of the ideal~$I$ by multiplication, so if we do a minimal extension search (Section \ref{sec:minext}) the extensions will often be `large jumps' (see Fig.~\ref{fig:lowertorsoenum}).
We can use Algorithm \ref{alg:minclosure} starting from $U$ and extending only by the elements from the ideal.
In practice, for the full transformation semigroups, this is a very useful trick.  The main general advantage of the Lower Torso Enumeration method is that no merge is required.

\begin{figure}[ht]
\def\S{(0,0) ellipse (5cm and 3cm)}
\def\I{(0,-2) ellipse (7cm and 3.5cm)}
\def\T{(-1,1)ellipse (.8cm and 1.4cm)}
\begin{tikzpicture}
\fill [color=gray9] \T;
\draw \T;
\draw (2.3,2.1) node (SI) {$S\setminus I$};
\draw (3.9,2.2) node (S) {$S$};
\draw (4,0) node (I) {$I$};
\draw (-.2,2.2) node (T) {$T$};
\draw (-1.1,.4) node (L) {$L$};
\draw (-1.1,2) node (U) {$U$};
\draw (-2.1,-1) node (T) {$x_1$};
\draw (2.1,0) node (T) {$x_2$};
\begin{scope}
\clip\S;
\draw\I;
\draw[very thick]\S;
\end{scope}
\begin{pgfonlayer}{background layer}
\draw [black] plot [smooth cycle] coordinates {(0,1.2)(-1.,1.8)(-4,-1)(-3,-2) (2.4,-1.6)(3,0)};
\fill [gray7] plot [smooth cycle] coordinates {(0,1.2)(-1.,1.8)(-4,-1)(-3,-2) (2.4,-1.6)(3,0)};
\draw [black] plot [smooth cycle] coordinates {(0,.5)(-1,2)(-2.6,-1) (-.4,-1)};
\fill [gray8] plot [smooth cycle] coordinates {(0,.5)(-1,2)(-2.6,-1) (-.4,-1)};
\end{pgfonlayer}
\end{tikzpicture}
\caption{Calculating lower torsos for $T$ by minimal extensions, first extending with $x_1$ then by $x_2$, elements from $I$. The idea is that often $|T| \ll  |\langle T\cup\{x_1\}\rangle| \ll |\langle T\cup\{x_1,x_2\}\rangle|$. The jumps in size are due to the upper torso acting on the elements of the ideal.}
\label{fig:lowertorsoenum}
\end{figure}
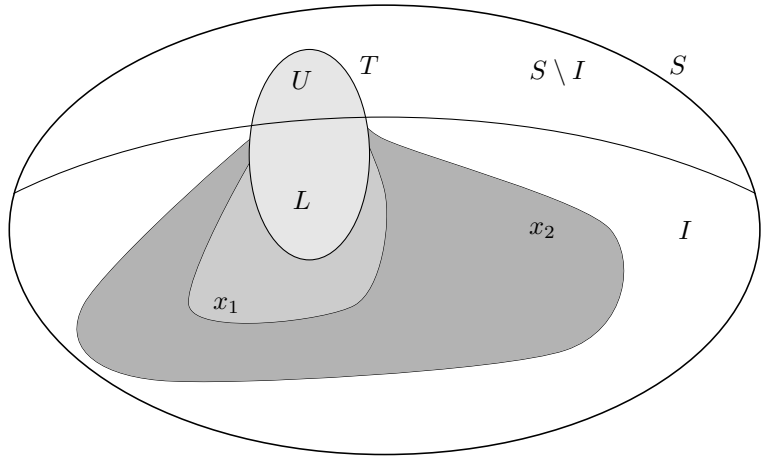

\section{Enumerating transformation semigroups of degree 2, 3 and 4}
\label{sec:fulltranssgp}

In order to enumerate all semigroups of degree $n$, we construct
all subsemigroups of the full transformation semigroup $\cT_n$. We use the ideal
structure to make the enumeration more efficient by making the calculation parallel according to the techniques discussed in Sections \ref{sec:idealparallel} and \ref{sec:lowertorso}.

Recall that the \emph{rank} of a transformation $t$ is $|\!\!\:\im(t)|$.
The ideal of $\cT_n$ containing all elements of rank at most $i$ is denoted by $K_{n,i}$.
The ideal structure of $\cT_n$ is a linear order of nested ideals:
$$\varnothing\subset K_{n,1}\subset K_{n,2}\subset\ldots\subset K_{n,n-1}\subset K_{n,n}=\cT_n.$$
The maximal proper ideal $K_{n,n-1}=\cT_n\setminus\cS_n$ is also called the \emph{singular transformation semigroup} of degree~$n$, and consists of all transformations but the permutations.

It is well known that $\Aut(\cT_n)$ is isomorphic to $\cS_n$, with
every automorphism induced by conjugating the elements of $\cT_n$ by a
permutation \cite{IN1972,SU1937}.
As such, we are primarily interested in calculating $\Sub_{\cS_n}(\cT_n)$, a set of conjugacy class representatives of the subsemigroups of $\cT_n$.  Note, however, that a pair of subsemigroups may be isomorphic but not conjugate. A separate backtrack search algorithm is needed to construct semigroup embeddings in general, then in particular we can find a set of isomorphism class representatives \cite{repcon}.

\subsection{Subsemigroups of $\cT_2$, the pen and paper case}
The semigroup $\cT_2$ has only four elements and consequently the brute-force search space size is only $2^4=16$.
It is an easy exercise to find all of its subsemigroups.
We order the elements lexicographically, 1=[1,1], 2=[1,2], 3=[2,1], 4=[2,2], and list the closed subarrays:
\renewcommand{\arraystretch}{0.3}
\begin{center}
\begin{tabular}{@{}c@{}c@{}c@{}c@{}}
1&1&4&4\\
1&2&3&4\\
1&3&2&4\\
1&4&1&4\\
\end{tabular},\ \ \ 
\begin{tabular}{@{}c@{}c@{}c@{}c@{}}
1&1&\color{gray9}4&4\\
1&2&\color{gray9}3&4\\
\color{gray9}1&\color{gray9}3&\color{gray9}2&\color{gray9}4\\
1&4&\color{gray9}1&4\\
\end{tabular},\ \ \ 
\begin{tabular}{@{}c@{}c@{}c@{}c@{}}
\color{gray9}1&\color{gray9}1&\color{gray9}4&\color{gray9}4\\
\color{gray9}1&2&\color{gray9}3&4\\
\color{gray9}1&\color{gray9}3&\color{gray9}2&\color{gray9}4\\
\color{gray9}1&4&\color{gray9}1&4\\
\end{tabular},\ \ \ 
\begin{tabular}{@{}c@{}c@{}c@{}c@{}}
\color{gray9}1&\color{gray9}1&\color{gray9}4&\color{gray9}4\\
\color{gray9}1&\color{gray9}2&\color{gray9}3&\color{gray9}4\\
\color{gray9}1&\color{gray9}3&\color{gray9}2&\color{gray9}4\\
\color{gray9}1&\color{gray9}4&\color{gray9}1&4\\
\end{tabular},\ \ \ 
\begin{tabular}{@{}c@{}c@{}c@{}c@{}}
1&1&\color{gray9}4&\color{gray9}4\\
1&2&\color{gray9}3&\color{gray9}4\\
\color{gray9}1&\color{gray9}3&\color{gray9}2&\color{gray9}4\\
\color{gray9}1&\color{gray9}4&\color{gray9}1&\color{gray9}4\\
\end{tabular},\ \ \ 
\begin{tabular}{@{}c@{}c@{}c@{}c@{}}
1&\color{gray9}1&\color{gray9}4&4\\
\color{gray9}1&\color{gray9}2&\color{gray9}3&\color{gray9}4\\
\color{gray9}1&\color{gray9}3&\color{gray9}2&\color{gray9}4\\
1&\color{gray9}4&\color{gray9}1&4\\
\end{tabular},\ \ \ 
\begin{tabular}{@{}c@{}c@{}c@{}c@{}}
1&\color{gray9}1&\color{gray9}4&\color{gray9}4\\
\color{gray9}1&\color{gray9}2&\color{gray9}3&\color{gray9}4\\
\color{gray9}1&\color{gray9}3&\color{gray9}2&\color{gray9}4\\
\color{gray9}1&\color{gray9}4&\color{gray9}1&\color{gray9}4\\
\end{tabular},\ \ \ 
\begin{tabular}{@{}c@{}c@{}c@{}c@{}}
\color{gray9}1&\color{gray9}1&\color{gray9}4&\color{gray9}4\\
 \color{gray9}1&2&3&\color{gray9}4\\
\color{gray9}1&3&2&\color{gray9}4\\
\color{gray9}1&\color{gray9}4&\color{gray9}1&\color{gray9}4\\
\end{tabular},\ \ \ 
\begin{tabular}{@{}c@{}c@{}c@{}c@{}}
\color{gray9}1&\color{gray9}1&\color{gray9}4&\color{gray9}4\\
\color{gray9}1&2&\color{gray9}3&\color{gray9}4\\
\color{gray9}1&\color{gray9}3&\color{gray9}2&\color{gray9}4\\
\color{gray9}1&\color{gray9}4&\color{gray9}1&\color{gray9}4\\
\end{tabular},\ \ \ 
\begin{tabular}{@{}c@{}c@{}c@{}c@{}}
\color{gray9}1&\color{gray9}1&\color{gray9}4&\color{gray9}4\\
\color{gray9}1&\color{gray9}2&\color{gray9}3&\color{gray9}4\\
\color{gray9}1&\color{gray9}3&\color{gray9}2&\color{gray9}4\\
\color{gray9}1&\color{gray9}4&\color{gray9}1&\color{gray9}4\\
\end{tabular}.

\end{center}

Using these we can draw the subsemigroup lattice (see Fig.~\ref{fig:T2subs}).  Note that the subsemigroups $\{1\}$ and $\{2\}$ are isomorphic but not conjugate.
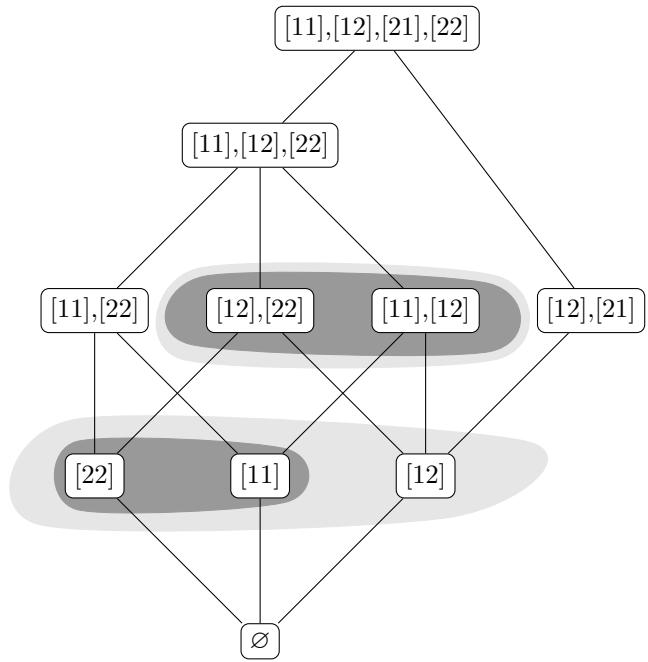
\begin{figure}[ht]
\begin{tikzpicture}
[align=center,node distance=2.2cm]
\tikzstyle{plain}=[fill=white,rounded corners=3pt, draw]

\draw node [plain] (1234) {[11],[12],[21],[22]};
\draw node [plain,below left of=1234] (124) {[11],[12],[22]};
\draw node [plain,below of=124] (24) {[12],[22]};
\draw node [plain,left of=24] (14) {[11],[22]};
\draw node [plain,right of=24] (12) {[11],[12]};
\draw node [plain,right of=12] (23) {[12],[21]};

\draw node [plain,below of=24] (1) {[11]};
\draw node [plain,left of=1] (4) {[22]};
\draw node [plain,right of=1] (2) {[12]};
\draw node [plain,below of=1] (empty) {$\varnothing$};

\draw  (1234) -- (124);
\draw (124) -- (12);
\draw (124) -- (24);
\draw (124) -- (14);
\draw  (1234) -- (23);
\path (23) edge (2);
\draw (1) -- (empty);
\draw (2) -- (empty);
\draw (4) -- (empty);
\draw (14) -- (4);
\draw (14) -- (1);
\draw (12) -- (2);
\draw (12) -- (1);
\draw (24) -- (4);
\draw (24) -- (2);
\begin{pgfonlayer}{background layer}
\filldraw [gray9] plot [smooth cycle] coordinates {(1.6,-4.4)(-2.6,-4.4) (-2.4,-3.2) (1.6,-3.3)};
\filldraw [gray6] plot [smooth cycle] coordinates {(1.5,-4.3)(-2.5,-4.2) (-2.3,-3.3) (1.5,-3.4)};
\filldraw [gray9] plot [smooth cycle] coordinates {(-4,-5.2) (2,-5.5) (1,-6.5)(-4.5,-6.6)};
\filldraw [gray6] plot [smooth cycle] coordinates {(-4,-5.5) (-1.2,-5.6) (-1.2,-6.3)(-4,-6.4)};

\end{pgfonlayer}

\end{tikzpicture}
\caption{The subsemigroup lattice of $\cT_2$. The horizontal levels correspond to classes of subsemigroups of the same size.  Dark (resp., light) grey blobs indicate nontrivial conjugacy (resp., isomorphism) classes.}
\label{fig:T2subs}
\end{figure}

The obvious subdivision of $\Sub(\cT_2)$ is according to the sizes of the subsemigroups (as in Fig.~\ref{fig:T2subs}).
It turns out that another way of partitioning the elements will also be important for higher degrees.
One big chunk of the subsemigroup lattice of $\cT_n$ is formed by the subsemigroups of the singular part, $\Sub(K_{n,n-1})$, and this has an order-isomorphic copy when we adjoin the identity of $\cT_n$ to each subsemigroup of $K_{n,n-1}$, denoted by $\Sub(K_{n,n-1})^\#$.
The remaining part is the set of subsemigroups that contain nontrivial permutations.
Since we have no problems with fully calculating and displaying $\Sub(\cT_2)$, this division has no significance, but can be visualized easily (see Fig.~\ref{fig:T2subsAlt}).
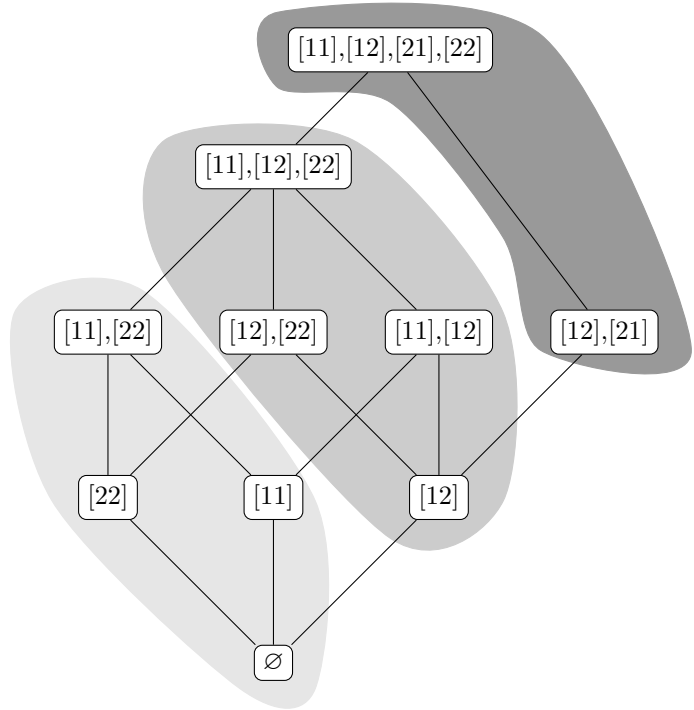
\begin{figure}
\begin{tikzpicture}
[align=center,node distance=2.2cm]
\tikzstyle{plain}=[fill=white,rounded corners=3pt, draw]

\draw node [plain] (1234) {[11],[12],[21],[22]};
\draw node [plain,below left of=1234] (124) {[11],[12],[22]};
\draw node [plain,below of=124] (24) {[12],[22]};
\draw node [plain,left of=24] (14) {[11],[22]};
\draw node [plain,right of=24] (12) {[11],[12]};
\draw node [plain,right of=12] (23) {[12],[21]};

\draw node [plain,below of=24] (1) {[11]};
\draw node [plain,left of=1] (4) {[22]};
\draw node [plain,right of=1] (2) {[12]};
\draw node [plain,below of=1] (empty) {$\varnothing$};

\draw  (1234) -- (124);
\draw (124) -- (12);
\draw (124) -- (24);
\draw (124) -- (14);
\draw  (1234) -- (23);
\path (23) edge (2);
\draw (1) -- (empty);
\draw (2) -- (empty);
\draw (4) -- (empty);
\draw (14) -- (4);
\draw (14) -- (1);
\draw (12) -- (2);
\draw (12) -- (1);
\draw (24) -- (4);
\draw (24) -- (2);
\begin{pgfonlayer}{background layer}
\filldraw [gray6] plot [smooth cycle] coordinates {(-1.5,-0.5)(-1.5,0.5)(2,0.2)(4,-4)(2,-4)(1.5,-2.5)(0,-0.7)};
\filldraw [gray8] plot [smooth cycle] coordinates {(-0.5,-1.2)(-3,-1.2)(-3,-3)(0,-6.5)(1.5,-6)(1.5,-3.5)};
\filldraw [gray9] plot [smooth cycle] coordinates {(-3.3,-3.2)(-5,-3.5)(-4.5,-6)(-2,-8.5)(-1,-8.5)(-1,-6)};

\end{pgfonlayer}

\end{tikzpicture}
\caption{The subsemigroup lattice of $\cT_2$. The subsemigroups of the singular part are indicated by the lowest light grey blob. The middle group is an order-isomorphic copy of the latter, with the identity of $\cT_2$ adjoined to each subsemigroup. The upper dark grey part consists of the subsemigroups containing nontrivial permutations. The size of the dark group appears to get smaller relative to the singular part for higher degrees.}
\label{fig:T2subsAlt}
\end{figure}

\subsection{Subsemigroups of $\cT_3$, the limits of brute-force}
The brute-force search space size for $\cT_3$ is $2^{3^3}=\text{134
  217 728}$, or approximately 134.2 million. Previous computational
investigations further restricted the search to the singular part
bringing it down to $\approx$2.1 million
\cite{T3enum1970, T3enum1991}.
In contrast, using the Minimal Extension method (Section~\ref{sec:minext}) together with the Equivalent Generators trick (Section \ref{sec:equivgen}), only 4344 subsets need to be checked to enumerate the 283 conjugacy classes. For all the 1299 subsemigroups, 25041 checks were required. This demonstrated efficiency of the graph search algorithm highlights the benefits of our approach.
The frequency distribution of the sizes of subsemigroups of $\Sub(\cT_3)$, as well as conjugacy and isomorphism classes, is as follows:
\begin{center}
\small
\renewcommand{\tabcolsep}{1pt}
\renewcommand{\arraystretch}{1}
\begin{tabular}{|l|c|c|c|c|c|c|c|c|c|c|c|c|c|c|c|c|c|c|c|c|c|c|c|c|c|c|c|c||c|}
\hline
Order&0&1&2&3& 4 & 5 & 6 & 7 & 8 & 9 & 10 & 11 & 12 & 13 & 14 & 15 & 16 & 17 & 18 & 19 & 20 & 21 & 22 & 23 & 24 & 25 & 26 & 27 & Total\\
\hline
\#subsemigroups&1&  10& 45& 86& 136& 192& 206& 186& 144& 109& 63& 51& 30& 9& 3& 9& 6& 6&&&& 1& 1& 3& 1&&& 1  &1299\\
\hline
\#conjugacy classes&1& \cellcolor{gray9}3& \cellcolor{gray9}10& \cellcolor{gray9}19& \cellcolor{gray9}28& \cellcolor{gray9}38&42&38&30&25&14&12&7&3&1&3&2&2& & &  &1&1&1&1& &  &1&283\\
\hline
\#isomorphism classes&1& \cellcolor{gray9}1& \cellcolor{gray9}5& \cellcolor{gray9}15& \cellcolor{gray9}24& \cellcolor{gray9}37&42&38&30&25&14&12&7&3&1&3&2&2& & &  &1&1&1&1& &  &1&267\\
\hline
with anti-isomorphism&1& 1& 5& \cellcolor{gray9}14& \cellcolor{gray9}23&37&42&38&30&25&14&12&7&3&1&3&2&2& & &  &1&1&1&1& &  &1&265\\
\hline
\end{tabular}
\normalsize
\end{center}
It is easy to see why $\cT_3$ has no subsemigroups of size 25 and 26: the biggest maximal subsemigroup is of order~24, since subsemigroups of order $\geq21$ correspond to those of the form $K_{3,2}\cup H$ where $H\leq \cS_3$. On the other hand, we have no such explanation for the missing orders
18, 19 and 20.
Observe also that isomorphism classes only break up into multiple conjugacy classes for low cardinalities (indicated by grey cells in the table). There is only one example of anti-isomorphism in $\cT_3$ (and a copy of it with identity included).

By applying the Minimal Generating sets method (Section \ref{sec:mingen}), we obtain the following rank value distribution of conjugacy
classes of $\Sub(\cT_3)$:
\begin{center}
\small
\renewcommand{\tabcolsep}{1pt}
\renewcommand{\arraystretch}{1}
\begin{tabular}{|c|c|c|c|c|c|c|c||c|}
\hline
rank&0&1&2&3& 4 & 5 & 6 & Total\\
\hline
\#subsemigroups &1 & 26& 201& 460& 410& 171& 30 &1299\\
\hline
\#conjugacy classes &1&  7& 46& 101& 85& 36& 7 &283\\
\hline
\#isomorphism classes &1& 4 & 39 & 96 &84 &36 &7 &267\\
\hline
\end{tabular}
\normalsize
\end{center}
In particular, the maximum rank of a subsemigroup of $\cT_3$ is $6$.  For example, the six transformations [1,1,3], $\id$, [2,1,1], [2,1,2], [2,2,3], [3,3,3] generate an 8-element semigroup (adding two more constants), but this semigroup cannot be generated by a smaller set.

\subsection{Subsemigroups of $\cT_4$, the need for parallelization}
Since $|\cT_4|=4^4=256$, the brute-force search space is already enormous; $2^{256}$ is a 78-digit number (by contrast, it is currently estimated that there are approximately $10^{80}$ atoms in the observable universe).
Therefore, the practical calculation of  $\Sub_{\cS_4}(\cT_4)$ requires the strategy  of cutting $K_{4,3}$ into two parts for doing the search for subsemigroups in parallel.
The exact algorithm for enumerating $\Sub_{\cS_4}(\cT_4)$ is described in the following six steps.
\begin{enumerate}
\item Calculate $\Sub_{\cS_4}(K_{4,3}/K_{4,2})$ by the minimal extension algorithm (Section \ref{sec:minext}). There are \mbox{10 002 390} conjugacy classes, slightly more than 10 million.
\item \emph{In parallel}, enumerate all lower torsos for all the upper torsos derived from  $\Sub_{\cS_4}(K_{4,3}/K_{4,2})$ with the Lower Torso limited enumeration method (Section \ref{sec:lowertorso}). This gives $\Sub_{\cS_4}(K_{4,3})$, with  65 997 018 conjugacy classes. The calculation is truly parallel since the upper torsos always differ, so there is no need for merging the elements. (The subsemigroups from $\Sub_{\cS_4}(K_{4,2})\subseteq \Sub_{\cS_4}(K_{4,3})$ are obtained in this step when extending the empty upper torso.)
\item To get the order-isomorphic copy of $\Sub_{\cS_4}(K_{4,3})$, we simply adjoin the identity to all subsemigroups: $\Sub_{\cS_4}(K_{4,3})^\#=\left\{S\cup\{\id\}\mid S\in \Sub_{\cS_4}(K_{4,3}) \right\}$.
\item To extend $\Sub_{\cS_4}(K_{4,3})$ to $\Sub_{\cS_4}(\cT_4)$, note that $\cT_4\setminus K_{4,3}=\cS_4$ is a sub(semi)group of $\cT_4$.  So we enumerate $\Sub_{\cS_4}(\cS_4)$ instead of $\Sub_{\cS_4}(\cT_4/K_{4,3})$, using the Minimal Extensions method. These are all closed upper torsos. This is a much easier subgroup enumeration problem.
\item \emph{In parallel}, find all lower torsos for all nontrivial subgroups in $\Sub_{\cS_4}(\cS_4)$. Let $P$ be the set of subsemigroups of $\cT_4$ with nontrivial permutations (including the subgroups as well). This part corresponds to the dark blob on Fig.~\ref{fig:T2subsAlt}.
Though the search space is the set of subsets (or in this case the subsemigroups) of $K_{4,3}$, the search is surprisingly quick.
This is due to the fact that a subgroup of $\cS_4$ acts on the singular part $K_{4,3}$, making each minimal extension into a huge jump, meaning that adding an extra generator yields a relatively large number of new elements.
In other words, we take each (conjugacy class representative) subgroup $1\neq G\leq \cS_4$ and look for subsemigroups of $K_{4,3}$ closed under the products with $G$.
Even a single nontrivial permutation makes the closure relatively big.
For instance, there are only 71147 lower torsos in $K_{4,3}$ for $\mathbb{Z}_2=\langle(1,2)\rangle$.
The total number of elements in $P$ is 75741.
\item Finally, we have $\Sub(\cT_4)=\Sub(K_{4,3})\cup \Sub(K_{4,3})^\# \cup P$, the set of subsemigroups of the singular part, its order-isomorphic copy with the identity adjoined to each subsemigroup, and the subsemigroups containing nontrivial permutations.
The total number of subsemigroups up to conjugacy is $|\Sub_{\cS_4}(\cT_4)|= 132 069 776$ and $|\Sub(\cT_4)|=3 161 965 550$. Note that the ratio of these two numbers is $\approx23.94$, almost the order of $\cS_4$, while $\Sub(\cS_4)/\Sub_{\cS_4}(\cS_4)=\frac{30}{11}\approx 2.72$.
\end{enumerate}

We note that if $M$ is an arbitrary finite monoid with group of units $G$, then there is a similar decomposition $\Sub(M)=\Sub(M\setminus G)\cup \Sub(M\setminus G)^\# \cup P$, where $P$ is defined analogously.

The size distribution of $\Sub(\cT_4)$ shows an interesting pattern (see Fig.~\ref{fig:SubT4SizeDistrib}).
For subgroups of a group, only the divisors of its order can have nonzero frequency values.
If we considered the size distribution of all subsets, the maximal (middle) binomial coefficient would define the peak value.
For $\Sub(\cT_4)$ the situation is far more involved.
The numbers are big and they give the impression of continuous change with several peaks.
The authors do not currently have an explanation for the shape of the distribution; a systematic study of the size classes is needed.

\begin{figure}[ht]
\includegraphics[width=\textwidth]{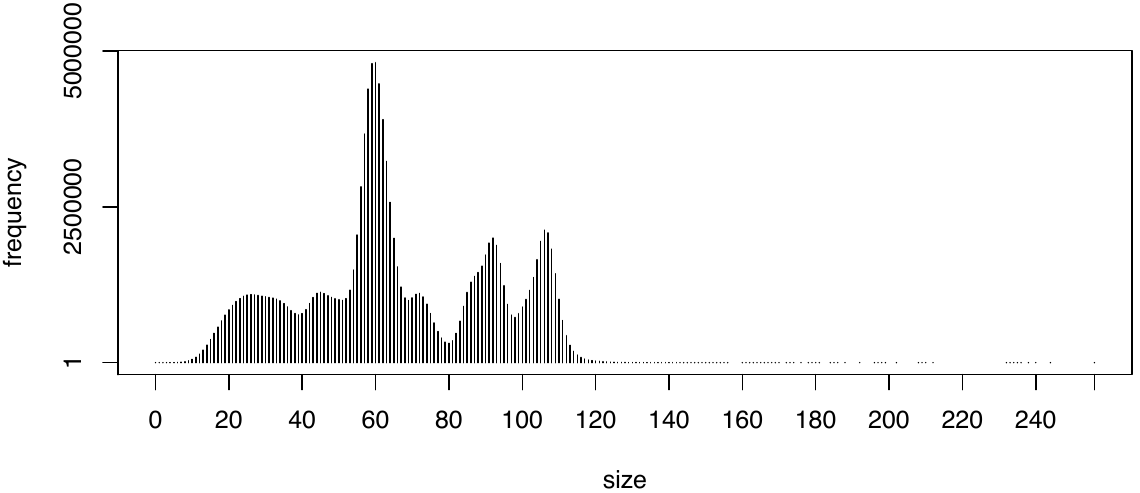}
\caption{The size distribution of $\Sub_{\cS_4}(\cT_4)$. The maximum is at size 60. There are 58 different size values with no subsemigroup (no corresponding dot in the figure); the values are 157, 158, 159, 171, 175, 177, 182, 183, 187, 189, 190, 191, 193, 194, 195, 200, 201, 203, 204, 205, 206, 207, 211, 213, 214, 215, 216, 217, 218, 219, 220, 221, 222, 223, 224, 225, 226, 227, 228, 229, 230, 231, 237, 239, 241, 242, 243, 245, 246, 247, 248, 249, 250, 251, 252, 253, 254, and 255.}
\label{fig:SubT4SizeDistrib}
\end{figure}

\subsection{Nilpotency}
Recall that a semigroup $S$ is \emph{nilpotent} if it has a zero element $0$ and $S^k=\{0\}$ for some $k\in\mathbb{N}$.
It is $k$-nilpotent if $k$ is the minimal such number.
(The empty semigroup is not nilpotent as it does not have a zero.)
To decide $k$-nilpotency algorithmically, in general we do not need to calculate the power~$S^k$.
We can take a random $k$-tuple of semigroup elements, evaluate it as a product and assume that this value is the zero element.
If we find any other $k$-tuple evaluating to a different value, then the semigroup is not $k$-nilpotent.
The worst case is when $S$ is indeed $k$-nilpotent, we end up checking all $k$-tuples.

It turns out that that there are only 4 nilpotent transformation semigroups on 3 points up to conjugacy.
The trivial monoid is 1-nilpotent and it can be realized by three different conjugacy classes: by the identity transformation, by a constant map, and by a conjugate of $[1,1,3]$.
The only $2$-nilpotent conjugacy class has the representative $\{[1,1,1],[1,1,2]\}$.

There are only 23 nilpotent subsemigroups of $\cT_4$ up to conjugacy; 5 of them are 1-nilpotent, 7 are 2-nilpotent and 11 are 3-nilpotent.
The biggest 3-nilpotent subsemigroup has 6 elements:
$\{[1,1,1,1]$, $[1,1,1,2]$, $[1,1,1,3]$, $[1,1,2,1]$, $[1,1,2,2]$, $[1,1,2,3]\}.$
Note also that the fraction of nilpotent conjugacy classes among all conjugacy classes of subsemigroups of $\cT_n$ ($n=1,2,3,4$) is $1$, $0.28$, $0.0141$, $1.74\times10^{-7}$.  In other words, at least for the first four values of $n$ (which may not be a large enough sample), nilpotent semigroups appear to become exceedingly rare.  This is in sharp contrast to the situation when enumerating semigroups by \emph{order}, which yields almost exclusively $3$-nilpotent semigroups as the order increases \cite{Distler2012}.  It is an intriguing question as to whether the fraction of nilpotents continues to decrease when enumerating by degree.  It would be very curious indeed if the two methods of ``slicing up'' the (infinite) set of finite semigroups (by order or by degree) led to the two seemingly contradictory (and ultimately meaningless, as we are simply decomposing a countably infinite set into two infinite subsets) intuitions that ``almost all semigroups are 3-nilpotent'' and ``almost all semigroups are non-nilpotent''.

\subsection{Symmetries}
Automorphism groups are of fundamental interest for any mathematical structure,
and for transformation semigroups we have another kind of symmetries
corresponding to relabelling points, the normalizer group.
For most semigroups in $\Sub_{\cS_4}(\cT_4)$ these groups are trivial
(Table \ref{tab:autnorm}).
\begin{table}
\begin{tabular}{|r|c|c|}
\hline
\#conjugacy classes & normalizer group & (order,index) \\
\hline
131435432 & 1 & (1,1) \\
 618518 & $\Z_2$ & (2,1) \\
  13652 & $\Z_2\times\Z_2$ & (4,2) \\
    952 & $\cS_3$ & (6,1) \\
    728 & $\Z_3$ & (3,1) \\
    464 & $D_4$ & (8,3) \\
     29 & $\cS_4$ & (24,12) \\
\hline
\end{tabular}
\vskip3pt
\begin{tabular}{|r|c|c|}
\hline
\#conjugacy classes & automorphism group & (order,index) \\
\hline
131286736 & 1 & (1,1) \\
 748946 & $\Z_2$ & (2,1) \\
29138 & $\Z_2\times\Z_2$ & (4,2) \\
 1296 & $D_4$ & (8,3) \\
 1144 & $\cS_3$ & (6,1) \\
969 & $\Z_2\times\Z_2\times \Z_2$ & (8,5) \\
717 & $\Z_3$ & (3,1) \\
296 & $D_6$ & (12,4) \\
182 & $Z_2\times D_4$ & (16,11) \\
 58 & $Z_2\times\Z_2\times \cS_3$ & (24,14) \\
 48 & $\cS_3\times\cS_3$ & (36,10) \\
 47 & $\cS_4$ & (24,12) \\
 44 & $\Z_2\times \cS_4$ & (48,48) \\
 29 & $\Z_2\times\Z_2\times \Z_2\times\Z_2$ & (16,14) \\
 28 & $D_4\times\cS_3$ & (48,38) \\
 22 & $\Z_2\times\cS_3\times\cS_3$ & (72,46) \\
 12 & $(\Z_2\times\Z_2\times \Z_2\times\Z_2)\rtimes\Z_2$ & (32,27) \\
8 & $\cS_4\times\cS_3$ & (144,183) \\
8 & $(\cS_3\times\cS_3)\rtimes\Z_2$ & (72,40) \\
8 & $D_4\times \cS_4$ & (192,1472) \\
8 & $\Z_2\times \cS_3\times \cS_4$ & (288,1028) \\
6 & $\Z_2\times ((\cS_3\times \cS_3)\rtimes \Z_2)$ & (144,186) \\
6 & $\Z_2\times\Z_2\times \cS_4$ & (96,226) \\
4 & $((\cS_3\times \cS_3)\rtimes \Z_2)\times \cS_3$ & (432,741) \\
4 & $(\Z_3\times\Z_3)\rtimes \Z_2$ & (18,4) \\
4 & $\Z_2\times\Z_2\times D_4$ & (32,46) \\
2 & $((\cS_3\times \cS_3)\rtimes \Z_2)\times \cS_4$ & (1728,47847) \\
2 & $\Z_2\times ((\Z_3\times\Z_3)\rtimes \Z_2)$ & (36,13) \\
2 & $\Z_2\times\Z_2\times \Z_2 \times \cS_3$ & (48,51) \\
1 & $\Z_2\times D_4\times \cS_3$ & (96,209) \\
\hline
\end{tabular}
\caption{Normalizer and automorphism groups of $\Sub_{\cS_4}(\cT_4)$, with
  \SmallGroup~library identification \cite{SmallGroupsLibrary}. }
\label{tab:autnorm}
\end{table}
\section{Summary and Conclusion}

We enumerated and classified all transformation semigroups up to degree 4.
The methods developed here, with more concentrated effort and computational power,  may be able to enumerate $\Sub(\cT_5)$ or at least the subsemigroups of some of its ideals/Rees quotients.
However, a better usage of the results would be to investigate the possibility of a more constructive theory of all transformation semigroups.
A comprehensive computational survey of more general diagram semigroups is also underway \cite{diagsgps}.
\begin{table}[ht]
\renewcommand{\arraystretch}{1}
\begin{tabular}{|c|r|r|r|r|}
\hline
 & \#subsemigroups & \#conjugacy classes & \#isomorphism classes & with anti-isomorphism\\
\hline
$\cT_0$ & 1  & 1 & 1 & 1\\
\hline
$\cT_1$ & 2  & 2 & 2 & 2\\
\hline
$\cT_2$ & 10  & 8& 7 & 7\\
\hline
$\cT_3$ & 1 299 & 283 & 267 & 265\\
\hline
$\cT_4$ & 3 161 965 550 & 132 069 776& 131 852 491 & 131 851 215\\
\hline
\end{tabular}
\caption{Number of subsemigroups of full transformation semigroups.}
\end{table}

\begin{table}[ht]
\renewcommand{\arraystretch}{1}
\begin{tabular}{|l|r|r|r|r|}
\hline
 & $\cT_1$ & $\cT_2$ & $\cT_3$ & $\cT_4$ \\
\hline
\#nilpotent & 1  & 2 & 4 & 23 \\
\hline
\#commutative & 1  & 4 & 18& 158\\
\hline
\#band & 1  & 5 & 41 & 1 503 \\
\hline
\#regular & 1  & 7 & 116 & 33 285\\
\hline
\#subsemigroups & 1  & 7 & 282 & 132 069 776\\
\hline
\end{tabular}
\caption{Number conjugacy classes that contain bands, commutative, and regular semigroups. (The empty semigroup is not counted.)}
\end{table}

\subsection*{Acknowledgements}
This work was partially supported by the NeCTAR Research Cloud, an
initiative of the Australian Government's Super Science scheme and the
Education Investment Fund; and by the EU project BIOMICS (contract number CNECT-ICT-318202).

\bibliography{../compsemi.bib}
\bibliographystyle{plain}

\end{document}